\documentclass[12pt]{amsart}
\usepackage{amsmath,amscd,amsthm,verbatim,alltt,amsfonts,array}
\usepackage[english]{babel}
\usepackage{latexsym}
\usepackage{amssymb}
\usepackage{euscript}
\usepackage{nicefrac}
\usepackage{graphicx}
\usepackage[all]{xy}

\newtheorem{Theorem}{Theorem}
\theoremstyle{plain}

\newtheorem{Corollary}{Corollary}

\newtheorem{Question}{Question}

\newtheorem{Example}{Example}

\newtheorem{Lemma}{Lemma}

\newtheorem{Proposition}{Proposition}
\newtheorem{Remark}{Remark}

\numberwithin{equation}{section}
\let\epsilon\varepsilon
\let\phi=\varphi
\let\kappa=\varkappa
\def\opn#1#2{\def#1{\operatorname{#2}}} 
\opn\sdepth{sdepth}
\opn\depth{depth}
\opn\chara{char} \opn\length{\ell} \opn\pd{pd} \opn\rk{rk}
\opn\projdim{proj\,dim} \opn\injdim{inj\,dim} \opn\rank{rank}
\opn\depth{depth} \opn\sdepth{sdepth} \opn\fdepth{fdepth}
\opn\grade{grade} \opn\height{height} \opn\embdim{emb\,dim}
\opn\codim{codim}  \opn\min{min} \opn\max{max}

\opn\Tr{Tr} \opn\bigrank{big\,rank}
\opn\superheight{superheight}\opn\lcm{lcm}
\opn\trdeg{tr\,deg}
\opn\reg{reg} \opn\lreg{lreg} \opn\ini{in} \opn\lpd{lpd}
\opn\size{size}

\opn\Spec{Spec} \opn\Supp{Supp} \opn\supp{supp} \opn\Sing{Sing}
\opn\Ass{Ass} \opn\Min{Min}
\opn\Ann{Ann} \opn\Rad{Rad} \opn\Soc{Soc}
\opn\Im{Im} \opn\Ker{Ker} \opn\Coker{Coker} \opn\Am{Am}
\opn\Hom{Hom} \opn\Tor{Tor} \opn\Ext{Ext} \opn\End{End}
\opn\Aut{Aut} \opn\id{id}  \opn\deg{deg}

\opn\nat{nat}
\opn\pff{pf}
\opn\Pf{Pf} \opn\GL{GL} \opn\SL{SL} \opn\mod{mod} \opn\ord{ord}
\opn\Gin{Gin} \opn\Hilb{Hilb}
\def\qed{\ifhmode\textqed\fi
      \ifmmode\ifinner\quad\qedsymbol\else\dispqed\fi\fi}
\def\textqed{\unskip\nobreak\penalty50
       \hskip2em\hbox{}\nobreak\hfil\qedsymbol
       \parfillskip=0pt \finalhyphendemerits=0}
\def\dispqed{\rlap{\qquad\qedsymbol}}

\begin{document}
\title{\bf Depth  in a pathological  case }

\author{ Dorin Popescu}

\thanks{The  support from  grant ID-PCE-2011-1023 of Romanian Ministry of Education, Research and Innovation is gratefully acknowledged.}

\address{Dorin Popescu,  Simion Stoilow Institute of Mathematics of Romanian Academy, Research unit 5,
 P.O.Box 1-764, Bucharest 014700, Romania}
\email{dorin.m.popescu@gmail.com}
\maketitle
\begin{abstract} Let $I$ be  a  squarefree monomial ideal of a polynomial algebra over a field minimally generated by $f_1,\ldots,f_r$ of  degree $ d\geq 1$, and a set $E$ of monomials  of degree $\geq d+1$. Let $J\subsetneq I$ be a squarefree monomial ideal generated in degree $\geq d+1$.  Suppose that  all squarefree monomials of $I\setminus (J\cup E)$ of degree $d+1$ are some least common multiples of $f_i$. If $J$ contains all least common multiples of two of $(f_i)$ of degree $d+2$ then $\depth_SI/J\leq d+1$ and Stanley's Conjecture holds for $I/J$.

 \noindent
  {\it Key words } : Monomial Ideals,  Depth, Stanley depth.\\
 {\it 2010 Mathematics Subject Classification: Primary 13C15, Secondary 13F20, 13F55,
13P10.}
\end{abstract}

\section*{Introduction}
Let $K$ be a field and $S=K[x_1,\ldots,x_n]$ be the polynomial $K$-algebra in $n$ variables. Let  $I\supsetneq J$ be  two    monomial ideals of $S$ and suppose that  $I$ is generated by some monomials of degrees $\geq d$   for some positive integer $d$.  After  a multigraded isomorphism we may assume either that $J=0$, or $J$ is generated in degrees $\geq d+1$.

 Suppose that $I \subset S$ is minimally generated by some  monomials $f_1,\ldots,f_r$ of degrees $d$,  and a  set $E$  of  monomials of degree $\geq d+1$.
 Let  $B$ (resp. $C$) be the set of   squarefree monomials of degrees $d+1$  (resp. $d+2$) of $I\setminus J$.   Let $w_{ij}$ be the least common multiple of $f_i$ and $f_j$, $i<j$ and set $W$ to be the set of all $w_{ij}$.
By \cite[Proposition 3.1]{HVZ} (see  \cite[Lemma 1.1]{P}) we have $\depth_S I/J\geq d$. It is easy to see that if $d=1$, $E=\emptyset$ and $B\subset W$ then $\depth_SI/J=d$ (see for instance \cite[Lemma 1.8]{P} and \cite[Lemma 3]{AP}). Attempts to extend this result were made in \cite[Proposition 1.3]{PZ}, \cite[Lemma 4]{AP}. However \cite[Example 1]{AP} (see here Example \ref{e2}) shows that for $d=2$, $E=\emptyset$ and $B\subset W$ it holds $\depth_SI/J=d+1=3$.

 If $B\cap (f_1,\ldots,f_r)\subset W$ we call $I/J$ a {\em pathological case}.
 It is the purpose of this paper  to show the following theorem.

\begin{Theorem} \label{t}   If  $B\cap (f_1,\ldots,f_r)\subset W$  and $C\cap W=\emptyset$ then $\depth_S I/J \leq d+1$.
\end{Theorem}

In particular, if $C\cap W=\emptyset$ then the so called Stanley's Conjecture  holds in the pathological case (see Corollary \ref{c}).
But why is important this pathological case? The methods used in \cite{PZ1}, \cite{AP}, \cite{P1} to show a weak form of Stanley's Conjecture when $r\leq 4$ (see \cite[Conjecture 0.1]{P1}) could be applied only   when $B\cap (f_1,\ldots,f_r)\not\subset W$, that is when $I/J$ is not  pathological. Thus the above theorem solves partially one of the  obstructions to prove this weak form. We believe that the condition $C\cap W=\emptyset$ could be removed from the above theorem. The proof of Theorem \ref{t} relies on Lemmas \ref{d'}, \ref{m} and Examples \ref{e4}, \ref{e}, \ref{e3} found after many computations with the Computer Algebra System SINGULAR \cite{DGPS}.

The above theorem hints a possible positive answer to the following question.

\begin{Question} Let $i\in [r-1]$. Suppose that $E=\emptyset$ and every squarefree monomial from $I\setminus J$ of degree $d+i$ is a least common multiple of $i+1$ monomials $f_j$. Then is it $\depth_SI/J\leq d+i$?
\end{Question}

We owe thanks to A. Zarojanu and a Referee, who noticed  some mistakes and a  gap in   Section 2 of  some previous versions of this paper.

\section{Depth and Stanley depth}

Suppose that $I$ is minimally generated by some squarefree monomials $f_1,\ldots,f_r$ of degree $ d$  for some $d\in {\mathbb N}$ and a set $E$ of some squarefree monomials of degree $\geq d+1$.  Let $C_3$ be the set of all $c\in C\cap (f_1,\ldots,f_r)$ having all degree $(d+1)$ divisors from $B\setminus E$ in $W$. In particular each monomial of $C_3 $ is the least common multiple of at least three of the $f_i$.

Next  lemma is closed to \cite[Lemma 1.1]{PZ}.
\begin{Lemma} \label{d} Suppose that $E=\emptyset$ and $\depth_SI/(J,b)=d$ for some $b\in B$. Then $\depth_SI/J\leq d+1$.
\end{Lemma}
\begin{proof} If there exists no $c\in C$ such that $b|c$ then we have $\depth_SI/J\leq d+1$ by \cite[Lemma 1.5]{PZ}.
Otherwise, in the exact sequence
$$0\to (b)/J\cap (b)\to I/J\to I/(J,b)\to 0$$
the first term has depth $\geq d+2$ because for a multiple $c\in C$ of $b$ all the variables of $c$ form a regular system. By hypothesis  the last term has depth $\geq d$ and so the middle one has depth $d$ too using  the Depth Lemma.
\hfill\ \end{proof}

We recall the following example from \cite{AP}.
\begin{Example} \label{e2}  {\em Let $n=5$, $r=5$, $d=2$, $f_1=x_1x_2$, $f_2=x_1x_3$, $f_3=x_1x_4$,  $f_4=x_2x_3$, $f_5=x_3x_5$ and $I=(f_1,\ldots,f_5)$,
$J=(x_1x_2x_5,x_1x_4x_5, x_2x_3x_4,x_3x_4x_5).$
It follows that  $B=\{x_1x_2x_3,x_1x_2x_4,x_1x_3x_4,x_1x_3x_5,x_2x_3x_5\}$ and so $s=|B|=r=5$. Note that $B\subset W$. A computation with SINGULAR when char $K=0$ gives $\depth_SI/J=\depth_SS/J=3$ and $\depth_SS/I=2$. Since depth depends on the characteristic of the field it follows in general only that $\depth_SS/J\leq 3$, $\depth_SS/I\leq 2$ using \cite[Lemma 2.4]{AE}. In fact $\depth_SI/J\leq d+1=3$ using \cite[Proposition 2.4]{Sh} because $q=|C|=2<r=5$. Note that choosing any $b\in B$ we  have $\depth_SI/(J,b)=2$ because the corresponding $s'<r $ and we may apply \cite[Theorem 2.2]{P}. But then $\depth_SI/J\leq 3$ by Lemma \ref{d}.}
\end{Example}
\begin{Example} \label{e4}  {\em In the above example set $I'=(f_1,\ldots,f_4)$, $J'=J\cap I'=$\\
$(x_1x_2x_5,x_1x_4x_5,x_2x_3x_4)$. Note that we have an injection $I'/J'\to I/J$ and so $\depth_SI'/J'>2$ because otherwise we get $\depth_SI/J=2$ which is impossible. Given $B',W'$ for $I'/J'$ we see that $B'\not\subset W'$  since $x_2x_3x_5\in B'\setminus W'$, that is $I'/J'$ is not anymore in the pathological case even this was the case of $I/J$.   We have $I/(J,f_5)\cong I'/(J',x_1x_3x_5,x_2x_3x_5)$. Using SINGULAR when char $K=0$ we see that $\depth_SI'/(J',x_1x_3x_5)=\depth_SI'/(J',x_1x_3x_5,x_2x_3x_5)=2$, $\depth_SI'/(J',x_2x_3x_5)=3$. It follows that always $\depth_SI'/(J',x_1x_3x_5),\depth_SI'/(J',x_1x_3x_5,x_2x_3x_5)\leq 2$ using \cite[Lemma 2.4]{AE}. These inequalities are in fact equalities because $I'$ is generated in degree $2$. Thus we cannot apply Lemma \ref{d} for $I',J'$, $b=x_2x_3x_5$ but we may apply this lemma for $I',J'$, $b'=x_1x_3x_5$ to get $\depth_SI'/J'\leq 3$.}
\end{Example}

\begin{Lemma} \label{d'} Suppose that $r>1$ and $\depth_SI/(J,f_r)=d$. Then $\depth_SI/J\leq d+1$.
\end{Lemma}
\begin{proof}  Let $B\cap (f_r)=\{b_1,\ldots,b_p\}$.
 As  $I/(J,f_r)$ has a squarefree, multigraded free resolution we see that only the
 components of squarefree degrees of
$$\Tor_S^{n-d}(K, I/(J,f_r))\cong H_{n-d}(x; I/(J,f_r))$$
are nonzero, the last module being the Koszul homology of $I/(J,f_r)$.
 Thus we may find
$$z =
\sum^{r-1}_{i=1} y_i f_i e_{[n]\setminus \supp f_i} \in K_{n-d}(x; I/(J,f_r)),$$
 $y_i \in K$
inducing a nonzero element in $H_{n-d}(x; I/(J,f_r))$.  Here we set $\supp f_i=\{t\in [n]:x_t|f_i\}$ and $e_A = \wedge_{j\in A} e_j$ for a
subset $A\subset [n]$.
We have
$$\partial z=\sum_{b\in B} P_b(y)be_{[n]\setminus \supp b}=\sum_{i=1}^p P_{b_i}(y)b_ie_{[n]\setminus \supp b_i},$$
where $P_b$ are linear homogeneous polynomials in $y$. Note that  $P_b(y)=0$ for all $b\not\in\{b_1,\ldots,b_p\}$. Choose $j\in \supp f_r$ and consider
$$z_j= \sum^{r-1}_{i=1, f_i\not\in (x_j)} y_i f_i e_{[n]\setminus  (\{x_j\}\cup\ \supp f_i)} \in K_{n-d-1}(x; I/J).$$
In $\partial z_j$ appear only terms of type $ue_A$, $j\not\in A$ with $|A|=n-d-1$ and $u=\Pi_{i\in [n]\setminus (A\cup \{j\})}x_i$. Thus terms of type $b_ie_{[n]\setminus \supp b_i}$, $i\in [p]$ are not present in $\partial z_j$ because $b_i\in (x_j)$. It follows that $\partial z_j=0$ and so $z_j$ is a cycle. Note that a cycle of $K_{n-d-1}(x; I/J)$ could contain also terms of type $ve_{A'}$ with $|A'|=n-d-1$ and $v=\Pi_{i\in [n]\setminus A'}x_i\in B$, but $z_j$ is just a particular cycle.

Remains to show that we may find $j$ such that $z_j$ is a nonzero cycle.
Suppose that $y_m\not =0$ for $m\in [r-1]$ and  choose  $j\in\supp f_r\setminus\supp f_m$. It follows that $z_j$ is a nonzero cycle because $y_mf_me_{[n]\setminus  (\{x_j\}\cup\ \supp f_m)}$ is present in $z_j$.
 Thus $\depth_SI/J\leq d+1$ by \cite[Theorem 1.6.17]{BH}.
\hfill\ \end{proof}

\begin{Remark}\label {r'} {\em Applying the above lemma to Example \ref{e4} we see that $\depth_SI/J\leq 3$ because $\depth_SI/(J,f_5)=2$.}
\end{Remark}

  Let $P_{I\setminus J}$  be the poset of all squarefree monomials of $I\setminus J$  with the order given by the divisibility. Let $ P$ be a partition of  $P_{I\setminus J}$ in intervals $[u,v]=\{w\in  P_{I\setminus J}: u|w, w|v\}$, let us say   $P_{I\setminus J}=\cup_i [u_i,v_i]$, the union being disjoint.
Define $\sdepth  P=\min_i\deg v_i$ and  the  {\em Stanley depth} of $I/J$ given by $\sdepth_SI/J=\max_{ P} \sdepth  P$, where $ P$ runs in the set of all partitions of $P_{I\setminus J}$ (see  \cite{HVZ}, \cite{S}).  Stanley's Conjecture says that  $\sdepth_S I/J\geq \depth_S I/J$.

In Example \ref{e2} we have $s=5<q+3=7$ and so it follows that $\depth_SI/J\leq d+1$ by \cite[Theorem 1.3]{P0} (see also \cite[Theorem 2]{AP}). Next example follows \cite[Example 1.6]{P1} and has $s=q+r$, $\sdepth_SI/J=d+2$ but $\depth_SI/J=d$.

\begin{Example}\label{e} {\em Let $n=12$, $r=11$, $f_1=x_{12}x_1$, $f_2=x_{12}x_2$, $f_3=x_{12}x_3$, $f_4=x_{12}x_4$, $f_5=x_{12}x_5$, $f_6=x_{12}x_6$, $f_7=x_6x_7$, $f_8=x_6x_8$, $f_9=x_6x_9$, $f_{10}=x_6x_{10}$, $f_{11}=x_6x_{11}$, $J=(x_7,\ldots,x_{11})(f_1,\ldots,f_5)+(x_1,\ldots,x_{5})(f_7,\ldots,f_{11})+f_6(x_9,\ldots,x_{11})$, $I=(f_1,\ldots,f_{11})$. We have $B=\{w_{ij}:1\leq i<j\leq 5\}\cup \{w_{kt}:6< k<t\leq 11\}\cup \{w_{i6}:i\in [8],i\not=6\}$, that is $s=|B|=27$.
 Let $c_1=x_6w_{12}$,   $c_2=x_6w_{23}$,  $c_3=x_6w_{34}$,  $c_4=x_6w_{45}$,  $c_5=x_6w_{15}$, $c_6=x_8w_{67}$, $c_7=x_9w_{78}$, $c_8=x_{10}w_{89}$, $c_9=x_{11}w_{9,10}$, $c_{10}=x_7w_{10,11}$,  $c_{11}=x_7w_{8,11}$, $c'_{13}=x_4w_{13}$, $c'_{14}=x_5w_{14}$,  $c'_{24}=x_6w_{24}$,
   $c'_{25}=x_3w_{25}$,  $c'_{35}=x_6w_{35}$. These are all monomials of $C$, that is $q=|C|=16$ and so $s=q+r$. The intervals $[f_i,c_i]$, $i\in [11]$ and $[w_{13},c'_{13}]$, $[w_{14},c'_{14}]$, $[w_{24},c'_{24}]$, $[w_{25},c'_{25}]$, $[w_{35},c'_{35}]$ induce a partition $P$ on $I/J$ with sdepth $4$.

We claim that $\depth_SS/J=2$. Indeed, let
$J'=(x_7,\ldots,x_{11})(f_1,\ldots,f_5)+(x_1,\ldots,x_{5})(f_7,\ldots,f_{11})=(x_{12},x_6)(x_7,\ldots,x_{11})(x_1,\ldots,x_{5})$. By \cite[Theorem 1.4]{Ap} we get $\depth_SS/J'=2=d$. Set $J_1=J'+(x_{12}x_6x_9)$, $J_2=J_1+(x_{12}x_6x_{10})$. We have $J=J_2+(x_{12}x_6x_{11})$. In the exact sequences
$$0\to (x_{12}x_6x_9)/(x_{12}x_6x_9)\cap J'\to S/J'\to S/J_1\to 0,$$
$$0\to (x_{12}x_6x_{10})/(x_{12}x_6x_{10})\cap J_1\to S/J_1\to S/J_2\to 0,$$
$$0\to (x_{12}x_6x_{11})/(x_{12}x_6x_{11})\cap J_2\to S/J_2\to S/J\to 0$$
the first terms have depth $\geq 5$. Applying the Depth Lemma by recurrence we get our claim.

Now we see that $\depth_SS/I=6$. Set $I_j=(f_1,\ldots,f_j)$ for $6\leq j\leq 11$. We have $I=I_{11}$, $I_6=x_{12}(x_1,\ldots,x_6)$ and $\depth_SS/I_6=6$. In the exact sequences
$$0\to (f_{j+1})/(f_{j+1})\cap I_j\to S/J_j\to S/I_{j+1}\to 0,$$
$6\leq j<11$ we have $(f_{j+1})\cap I_j=f_{j+1}(x_{12},x_7,\ldots,x_{j})$ and so $\depth_S(f_{j+1})/(f_{j+1})\cap I_j=12-(j-5)\geq 7$
for $6\leq j<11$. Applying the Depth Lemma by recurrence we get $\depth_SS/I_{j+1}=6$ for $6\leq j<11$ which is enough.

Finally using the Depth Lemma in the exact sequence
$$0\to I/J\to S/J\to S/I\to 0$$
 it follows $\depth_SI/J=2=d$.   }
 \end{Example}

The following lemma is the key in the proof of Theorem \ref{t} and its proof is given in the next section.
\begin{Lemma} \label{m} Suppose that  $E=\emptyset$, $C\subset  C_3$, $C\cap W=\emptyset$ and Theorem \ref{t} holds for $r'<r$. Then $\depth_SI/J\leq d+1$.
\end{Lemma}

\begin{Proposition} \label{p} Suppose that   $C\cap (f_1,\ldots,f_r)\subset  C_3$, $C\cap W=\emptyset$ and  Theorem \ref{t} holds for  $r'<r$. Then $\depth_SI/J\leq d+1$.
\end{Proposition}
\begin{proof}
Suppose that $E\not =\emptyset$, otherwise  apply Lemma \ref{m}. Set $I'=(f_1,\ldots,f_r)$, $J'=J\cap I'$. In the exact sequence
$$0\to I'/J'\to I/J\to I/(I',J)\to 0$$
the last term is isomorphic to something generated by $E$ and so its depth is $\geq d+1$. The first term satisfies the conditions of Lemma \ref{m} which gives $\depth_SI'/J'\leq d+1$. By the Depth Lemma we get $\depth_SI/J\leq d+1$ too.
\hfill\ \end{proof}

{\bf Proof of Theorem \ref{t}}

Apply induction on $r$. If $r<5$ then $B\cap (f_1,\ldots,f_r)\subset W$ implies   $|B\cap (f_1,\ldots,f_r)|<2r$ and so $\sdepth_SI/J\leq d+1$ and even $\depth_SI/J\leq d+1$ by \cite[Proposition 2.4]{Sh} (we may also apply \cite[Theorem 0.3]{P1}). Suppose that  $r\geq 5$. Since all divisors of a monomial $c\in C\cap (f_1,\ldots,f_r)$ of degrees  $d+1$ are in $B$, they are also in $W$ by our hypothesis.   Thus  $C\cap (f_1,\ldots,f_r)\subset C_3$ and we may apply Proposition \ref{p} under induction hypothesis.  $\hfill\ \square$

\begin{Corollary} \label{c} Suppose that $B\cap (f_1,\ldots,f_r)\subset W$ and  $C\cap W=\emptyset$. Then $\depth_SI/J\leq \sdepth_SI/J$, that is the Stanley Conjecture holds for $I/J$.
\end{Corollary}
\begin{proof} If $\sdepth_SI/J=d$ then apply \cite[Theorem 4.3]{P}, otherwise apply Theorem \ref{t}.
\hfill\ \end{proof}

\section{Proof of Lemma \ref{m}.}

 We may suppose that $B\subset W$ because each monomial of $B$ must divide a monomial of $C$, otherwise we get $\depth_SI/J\leq d+1$ by \cite[Lemma 1.5]{PZ}. Then we may suppose that $B\subset \cup_i \supp f_i$ and we may reduce to the case when $[n]=\cup_i \supp f_i$ because then $\depth_SI/J=\depth_{\tilde S} (I\cap {\tilde S})/(J\cap {\tilde S})$ for ${\tilde S}=K[\{x_t:t\in \cup_i \supp f_i\}]$.

  On the other hand, we may suppose that for each $i\in [r]$ there exists $c\in C$ such that $f_i|c$, otherwise we may apply again \cite[Lemma 1.5]{PZ}. Since  $c\in C_3$, let us say $c$ is the least common multiple of $f_1,f_2,f_3$ we see that at least,let us say, $w_{12}\in B$. Then  $f_i\in (u_1)$, $i\in [2]$ for some monomial $u_1=(f_1f_2)/w_{12}$ of degree $d-1$.

 We may assume that $f_i\in (u_1)$ if and only if $i\in [k_1]$ for some $2\leq k_1\leq r$.
Set $U_1=\{f_1,\ldots,f_{k_1}\}$.
We also assume that
$$\{u_i:i\in [e]\}=\{u:u=gcd(f_i,f_j), \deg u=d-1, i\not = j\in [r]\},$$
and define
$$U_i=\{f_j:f_j\in (u_i), j\in [r]\}$$
for each $i\in [e]$.
 Since  each $f_t\in U_i$  divides a certain $c\in C$ we see from our construction that  there exist $f_p,f_l\in U_i$ such that $w_{tp},w_{tl}\in B$.  Note that if $|U_i\cap U_j|\geq 2$ then we get $u_i=u_j$ and so $i=j$. Thus $|U_i\cap U_j|\leq 1$ for all $i,j\in [e]$, $i\not =j$.

  Suppose that  $w_{ij}\in J$ for all $i\in [k_1]$ and for all $j>k_1$,
 let us say $f_i=u_1x_i$ for $i\in [k_1]$.  Set $I'=(f_1,\ldots,f_{k_1})$, $J'=I'\cap J$  and
 $\hat S=K[\{x_i:i\in [k_1]\cup \supp u_1\}]$. Then
 $\depth_SI'/J'=\depth_{\hat S} I'\cap \hat S/J'\cap \hat S=\deg u_1+\depth_{\hat S} (x_1,\ldots,x_{k_1})\hat S/(J':u_1)\cap \hat S=d$. Thus $\depth_SI/J=d$ by the Depth Lemma applied to the exact sequence
$$0\to I'/J'\to I/J\to I/(I',J)\to 0,$$
since the last term has depth $\geq d$  being generated by squarefree monomials of degree $\geq d$. In particular, $\depth_SI/J=d$  if $e=1$.
If $e>1$ we may assume that for each $i\in [e]$ there exists $j\in [e]$ with $U_i\cap U_j\not =\emptyset$.
\begin{Example}{\em Back to Example \ref{e} note that we may take $u_1=x_{12}$, $u_2=x_6$ and $U_1=\{f_1,\ldots,f_6\}$, $U_2=\{f_6,\ldots,f_{11}\}$.}
\end{Example}
\begin{Lemma}\label{l2} Suppose that $e\geq 2$ and $f_r\in U_e$. Let $I'$ be the ideal generated by all $f_k\in U_e\setminus \{f_r\}$. If
$\depth_SI/(J,I',f_r)\geq d+1$ and
there exists $t$ with $f_t\in (\cup_{i=1}^{e-1} U_i)\setminus U_e$ such that $w_{rt}\in B\setminus I'$ then $\depth_SI/(J,I')=d+1$.
\end{Lemma}
\begin{proof}
Using the Depth Lemma applied to the exact sequence
$$0\to (f_{r})/(f_{r})\cap (J,I')\to I/(J,I')\to I/(J,I',f_r)\to 0$$
 we see that it is enough to show that
 $\depth_S (f_{r})/(f_{r})\cap (J,I')=  d+1$.
 By our hypothesis the squarefree monomials from $(f_{r})\setminus (J,I')$ have the form $w_{rt'}$ for some $t'$ with
 $f_{t'}\in (\cup_{i=1}^{e-1} U_i)\setminus U_e$ and   $w_{rt'}\in B\setminus I'$.

 Next we will describe the above set of monomials.
 If there exists no $U_{l}$ containing $f_t$, $f_{t'}$ then their contributions to $(f_{r})/(f_{r})\cap (J,I')$ consist in two different monomials $w_{rt}$,  $w_{rt'}$. Otherwise, we must have $f_r=x_kx_pv$, $f_t=x_kx_mv$ and $f_{t'}=x_px_mv$  for some different $k,m,p\in [n]$ and one monomial $v$ of degree $d-2$. Thus $w_{rt}= w_{rt'}$ and the contributions of $f_t$, $f_{t'}$ consist in just one monomial.  Let $A$ be the set of all $f_k \in  (\cup_{i=1}^{e-1} U_i)\setminus U_e$ such that $w_{rk}\in B\setminus I'$  and  define an equivalence relation on $A$ by $f_t\sim f_{t'}$ if $f_t, f_{t'}\in U_i$ for some $i\in [e-1]$.
For some $f_t$ from an equivalence class of $A/\sim$ we have $w_{rt}=x_{\gamma_t}f_{r}$ for one $\gamma_t\in [n]$.  Let $\Gamma$ be the set of all these variables $x_{\gamma_t}$ for which $w_{rt}\not\in (J,I')$. For two $x_{\gamma_t}$, $x_{\gamma_{t'}}$ corresponding to different classes we have
$w_{tt'}=x_{\gamma_t}x_{\gamma_{t'}}f_{r}$ since $f_t$, $f_{t'}$ are not in the same equivalence class. Thus $w_{tt'}\in J$ because otherwise $w_{tt'}\in C$ which is impossible by our hypothesis. Let $Q\subset K[\Gamma]$  be the ideal  generated by all squarefree quadratic monomials. The multiplication by $f_{r}$ gives a bijection between $K[\Gamma]/Q$ and $(f_{r})/(f_{r})\cap (J,I')$ because each squarefree monomial of $(B\cap(f_r))\setminus I'$ has the form $w_{rt}=f_rx_{\gamma_t}$ for some $t$, $x_{\gamma_t}$ being in $\Gamma$. Then $\depth_S (f_{r})/(f_{r})\cap (J,I')=d+\depth_{K[\Gamma]}K[\Gamma]/Q=d+1 $ since the variables of $f_r$ form a regular sequence for $f_{r})/(f_{r})\cap (J,I')$ (in the squarefree frame).
 Note that if $A/\sim$ has just one class of equivalence containing some $f_t$ with $w_{tr}\in B$ then $|\Gamma|=1$, $Q=0$
and also it holds $\depth_{K[\Gamma]}K[\Gamma]/Q=1$.
\hfill\ \end{proof}

\begin{Remark}\label{l2'} {\em In the notations of the above lemma suppose that $w_{rt}\in (J,I')$ for all $t$ with $f_t\in (\cup_{i=1}^{e-1} U_i)\setminus U_e$.
Then there exists no squarefree monomial of degree $d+1$ in $(f_r)\setminus (J,I')$ and so $\depth_S(f_r)/(f_r)\cap (J,I')=d$. It follows that $\depth_SI/(J,I')=d$ too.}
\end{Remark}

\begin{Lemma} \label{l3} Suppose that $e\geq 2$.  If
$\depth_SI/(J,(U_e))\leq  d+1$ then $\depth_SI/J\leq d+1$.
\end{Lemma}
\begin{proof} Suppose that $U_e\supset \{f_{k+1},\ldots,f_r\}$ for some $k\leq r$.
 Let $I'_k$ be the ideal generated by all $f_t\in U_e\setminus \{f_{k+1},\ldots,f_r\}$.
 We  claim that $\depth_S I/ (J,I'_k)\leq   d+1$. Apply induction on $r-k$, the case $k=r-1$ being done in Lemma \ref{l2} and Remark \ref{l2'}.
 Assume that $r-k>1$ and note that $I'_{k+1}=(I'_k,f_{k+1})$. By induction hypothesis we have $\depth_S I/(J,I'_{k+1})\leq d+1$. If $\depth_S I/(J,I'_{k+1})=d$
 then we get $\depth_SI/(J,I'_k)\leq d+1$ by Lemma \ref{d'}. If $\depth_S I/(J,I'_{k+1})=d+1$ we get again  $\depth_SI/(J,I'_k)\leq d+1$ by Lemma \ref{l2} and Remark \ref{l2'}. This proves our claim.

 Now choose $r-k$ maxim, let us say  $U_e= \{f_{k+1},\ldots,f_r\}$. Then $I'_k=0$ and so we get $\depth_SI/J\leq d+1$.
\hfill\ \end{proof}

\begin{Example} \label{e3} {\em Let $n=6$, $r=8$, $d=2$, $f_1=x_1x_2$, $f_2=x_1x_3$, $f_3=x_1x_4$,  $f_4=x_2x_3$, $f_5=x_3x_5$, $f_6=x_2x_6$, $f_7=x_3x_6$,   $f_8=x_4x_6$, and $I=(f_1,\ldots,f_8)$,
$$J=(x_1x_2x_5,x_1x_3x_6,x_1x_4x_5,x_1x_4x_6, x_2x_3x_4, x_2x_5x_6,  x_3x_4x_5, x_3x_5x_6, x_4x_5x_6).$$
It follows that $U_1=\{f_1,\ldots,f_3\}$, $U_2=\{f_2,f_4,f_5,f_7\}$, $U_3=\{f_6,f_7,f_8\}$, $U_4=\{f_1,f_4,f_6\}$, $U_5=\{f_3,f_8\}$, $u_1=x_1$, $u_2=x_3$, $u_3=x_6$, $u_4=x_2$, $u_5=x_4$.  Let ${\tilde S}=K[x_1,\ldots,x_5]$, ${\tilde J}\subset {\tilde I}\subset {\tilde S}$ be the corresponding ideals given in Example \ref{e2}. For  $I'=(U_3)$ we have $I/(J,I')\cong {\tilde I}S/({\tilde J},x_6{\tilde I})S$. Thus  $\depth_SI/(J,I')=\depth_{\tilde S}{\tilde I}/{\tilde J}\leq 3=d+1$ by Example \ref{e2}  and so using Lemma \ref{l3} we get $\depth_SI/J\leq 3$ too.}
\end{Example}

\begin{Proposition}\label{pr} If $\cap_{i\in [e]}U_i\not=\emptyset$ then $\depth_SI/J\leq d+1$.
\end{Proposition}
\begin{proof} As we have seen $\depth_SI/J=d$ if $e=1$. Assume that $e>1$ and let us say $f_r\in \cap_{i\in [e]}U_i$. Set $I'=(U_e)$. In $I/(J,I')$ we have $(e-1)$ disjoint $U'_i=U_i\setminus \{f_r\}$, $i\in [e-1]$. It follows that $\depth_SI/(J,I')=d$ and so $\depth_SI/J\leq d+1$ by Lemma \ref{l3}.
\hfill\ \end{proof}

{\bf Proof of Lemma \ref{m}.}

 If $\cap_{i\in [e]}U_i\not=\emptyset$ then apply the above proposition. Otherwise,  suppose that
$f_r\in U_j$ if and only if $1\leq j<k$ for some $1<k\leq e$. Set $I'=(U_{k},\ldots,U_e)$. Applying again the above proposition we get $\depth_SI/(J,I')\leq d+1$.
Set $L_i=\cup_{j\geq i} U_j$.  Since  $I'=(L_{k}) $ and $\depth_SI/(J,L_{k})\leq d+1$  we see that   $\depth_SI/(J,L_{k+1})\leq d+1$ by Lemma \ref{l3}. Using by recurrence Lemma \ref{l3} we get  $\depth_SI/J=\depth_SI/(J,L_{e+1})\leq d+1$ since $L_{e+1}=\emptyset$.
 $\hfill\ \square$

It is not necessary  to assume in  Proposition \ref{pr} that $C\cap W=\emptyset$ because anyway this follows as shows the following lemma.

\begin{Lemma} \label{l4} If $\cap_{i\in [e]}U_i\not=\emptyset$ then $C\cap W=\emptyset$.
\end{Lemma}
\begin{proof} Clearly we may suppose that $e>1$. Let $f_r\in \cap_{i\in [e]}U_i$. Suppose that $w_{tt'}\in C$ for some $t\in U_i$, $t'\in U_j$. Then  $i\not =j$ and let us say $f_t=u_ix_1$, $f_{t'}=u_jx_2$ and $f_r=u_ix_3=u_jx_4$. It follows that $u_i=x_4v$, $u_j=x_3v$ for some monomial $v$ and so $f_t=x_1x_4v$, $f_{t'}=x_2x_3v$, $f_r=vx_3x_4$. Note that $f_tx_2\in B$ and so must be of type $w_{tt''}$ for some $t''\in [r]$. It follows that $f_t,f_{t''}\in U_k$ for some $k\in [e]$. By hypothesis $f_r\in U_k$ and so $k=i$ because otherwise $|U_k\cap U_i|>1$ which is impossible. Since $f_tx_2\in B$ we see that $x_2\not\in \supp u_i$ and it follows that $f_{t''}=x_2x_4v$. Therefore, $w_{t't''}=x_2x_3x_4v\in B$ and so $f_{t'},f_{t''}\in U_p$ for some $p$.  This is not possible because otherwise one of $|U_p\cap U_i|$, $|U_p\cap U_j|$, $|U_i\cap U_j|$ is  $\geq 2$. Contradiction!
\hfill\  \end{proof}

 \end{document}